\documentclass[12pt]{amsart}

\usepackage{amsmath,amssymb,amsfonts,fullpage,amssymb,amsthm}

\newtheorem{thm}{Theorem}

\newtheorem{rem}{Remark}

\begin{document}

\title{Homogeneous Dual Ramsey Theorem}

\author{\textsc{Jos\'e G. Mijares}}\address{Department of Mathematics. California State University, Los Angeles. 5151 State University Drive 
Los Angeles, CA 90032. goyomijares@gmail.com}

\maketitle

\begin{abstract}
For positive integers $k < n$ such that $k$ divides $n$, let $(n)^k_{\hom}$ be the set of homogeneous $k$-partitions of $\{1, \dots, n\}$,  that is, the set of partitions of $\{1, \dots, n\}$ into $k$ classes of the same cardinality. In \cite{KST} the following question (Problem 7.3 in \cite{KST}) was asked: 

\medskip

Is it true that given positive integers $k < m$ and $N$ such that $k$ divides $m$, there exists a number $n>m$ such that $m$ divides $n$, satisfying that for every coloring $(n)^k_{\hom}=C_1\cup\dots\cup C_N$ we can choose $u\in (n)^m_{\hom}$ such that $\{t\in (n)^k_{\hom}: t\mbox{ is coarser than } u\}\subseteq C_i$ for some $i$?

\medskip

 In this note we give a positive answer to that question in Theorem \ref{hdrt} below. This result turns out to be a homogeneous version of the finite Dual Ramsey Theorem of Graham-Rothschild \cite{grahro}. As explained in \cite{KST}, our result also proves that the class $\mathcal{OMBA}_{\mathbb Q_2}$ of naturally ordered finite measure algebras with measure taking values in the dyadic rationals has the Ramsey property.
\end{abstract}

\bigskip


\section{Homogeneous Dual Ramsey Theorem}

\subsection{\bf Homogeneous partitions.} Let $k<n$ be positive integers. As customary,  write $k | n$ if $k$ divides $n$, and in that case let 

 \begin{equation}(n)^k_{\hom} =  \mbox{the set of all homogeneous } k\mbox{-partitions of } \{1, \dots, n\};
  \end{equation}  that is, the set of all the partitions of $\{1, \dots, n\}$ into $k$ classes of the same cardinality. If $t$ and $u$ are partitions of $\{1, \dots, n\}$, not necessarily homogeneous, then we say that $t$ is \textit{coarser than} $u$ if every class of $u$ is a subset of some class of $t$. Denote by $(u)^k_{\hom}$ the set of all homogeneous $k$-partitions of $\{1, \dots, n\}$ that are coarser than $u$.

\bigskip

\begin{rem}
In \cite{KST}, the notation $EQ_k^{\hom}$ is used to denote $(n)^k_{\hom}$. 
\end{rem}

The following is our main result.

\begin{thm}[Homogeneous Dual Ramsey Theorem]\label{hdrt}
Let integers $k < m$ and $N$ be given, with $k|m$. There exists a positive integer  $n>m$ such that $m|n$, satisfying that for every coloring $(n)^k_{\hom}=C_1\cup\dots\cup C_N$ there exists $u\in (n)^m_{\hom}$ such that $(u)^k_{\hom}\subseteq C_i$, for some $i\in\{1, \dots, N\}$.
\end{thm}


We will prove Theorem \ref{hdrt} in Section \ref{proof} below. In our proof, we will make use of the infinite Dual Ramsey Theorem due to Carlson and Simpson \cite{CarSim}, which we will state in the next subsection.

 \subsection{\bf Dual Ramsey Theorem.}  We will essentially go back to the notation introduced in \cite{CarSim}.  Let $\mathbb{N} = \{1, 2, 3, \dots \}$ be the set of positive integers and let $(\omega)^{\omega}$ be the set of all the infinite partitions $A = \{A_i : i\in\mathbb{N}\}$ of $\mathbb{N}$ into infinite classes such that 
 \begin{equation}
 i<j \Longrightarrow \min (A_i) < \min (A_j).
 \end{equation}
 Given $A, B\in (\omega)^{\omega}$, we say that $A$ is \textit{coarser} than $B$ if every class in $B$ is a subset of some class in $A$. Pre-order $(\omega)^{\omega}$ as follows:
 
 \begin{equation}A\leq B \Longleftrightarrow \ A\ \mbox{is coarser than}\ B. \end{equation}
 Likewise, given a positive integer $k$, let $(\omega)^k$ be the set of all the partitions \linebreak $X = \{X_i : 1\leq i\leq k\}$ of $\mathbb N$ into $k$ infinite classes satisfying 
 
 \begin{equation}
 i<j \Longrightarrow \min (X_i) < \min (X_j),\ \ \mbox{ for  } 1\leq i, j\leq k.
 \end{equation}

 \medskip
 
\noindent  Given $X,Y\in (\omega)^k$ and $A\in (\omega)^{\omega}$, define ``$X$ is \textit{coarser} than $Y$"  (resp. ``$X$ is \textit{coarser} than $A$")  in the same way as above; and write $X\leq Y$ (resp. $X\leq A$). For $A\in (\omega)^{\omega}$, let $(A)^k = \{X\in (\omega)^k : X\leq A\}$.
 
 \medskip
 
 \noindent We will regard the sets $(\omega)^{\omega}$ and $(\omega)^k$ as topological spaces with the usual topologies as defined in \cite{CarSim}.
 
\begin{thm}[Infinite Dual Ramsey Theorem; Carlson-Simpson \cite{CarSim}]\label{borel parameter words ramsey} Let $k$ be a positive integer. For every finite Borel-measurable coloring of $(\omega)^k$ there exists $A\in(\omega)^{\omega}$  such that $(A)^k$ is monochromatic.
\end{thm}

For positive integers $k$ and $n$, let 

 \begin{equation} (n)^k = \mbox{ the set of all the } k\mbox{-partitions of } n, \end{equation} 
 
\noindent  i.e., partitions of $\{1,\dots, n\}$ into $k$ classes. For every such partition $u$, denote by $(u)^k$ the set of all $k$-partitions of $\{1, \dots, n\}$ that are coarser than $u$. Also, let $(<\omega)^k = \bigcup_{n\in\mathbb{N}}(n)^k$. 
 

\medskip


Theorem \ref{hdrt} is a homogeneous version of the following well-known result due to Graham and Rothschild, which can be obtained as a corollary of Theorem \ref{borel parameter words ramsey}:

\begin{thm}[Dual Ramsey Theorem; Graham-Rothschild \cite{grahro}]\label{drt}
	Let integers $k< m$ and $N$ be given. There exists a number $n>m$ satisfying that for every coloring $(n)^k=C_1\cup\dots\cup C_N$ there exists $u\in (n)^m$ such that $(u)^k\subseteq C_i$ for some $i\in\{1,\dots, N\}$.
\end{thm}

\medskip


 \section{Proof of Theorem \ref{hdrt}}\label{proof}

Let $A =  \{A_i : i\in\mathbb{N}\}\in (\omega)^{\omega}$ be given. We will borrow some notation from  \cite{todo}. Define 
 \begin{equation}r(0, A) = \emptyset\ \  \&\ \ r(i,A) = \{A_j\cap\{1, \dots,  \min (A_i)\} : 1\leq j\leq i\}; \ i>0. \end{equation}

\noindent Note that $r(i,A)$ is a partition of $\{1, \dots,  \min (A_i)\}$. We think of it as the $i$-th approximation of $A$. If a finite partition $b$ is the $i$-th approximation of some $B\leq A$, we write $b\in(<\omega)^{i}\upharpoonright A$.

It will be useful to understand elements $A= \{A_i : i\in\mathbb{N}\}\in  (\omega)^{\omega}$ as surjective functions\linebreak $A: \mathbb N\rightarrow \mathbb N$ where $A_i = A^{-1}(\{i\})$. Likewise, elements  $X = \{X_i : 1\leq i\leq k\}\in  (\omega)^k$ will be as well understood as surjective functions $X: \mathbb N\rightarrow \mathbb \{1, \dots k\}$ where $X_i = X^{-1}(\{i\})$, for $1\leq i\leq k$.  In the same spirit, we will regard partitions $u$ of a finite set $F\neq\emptyset$ into $k$ classes as surjections $u : F\rightarrow \{1,\dots, k\}$. We will shift between understanding partitions as surjective functions and sets of disjoint classes throughout the rest of this article. Now, keep that in mind and fix a positive integer $k$ for a while. We will borrow some notation and ideas from \cite{DM}. Given $A\in  (\omega)^{\omega}$, let $\pi: (\omega)^{\omega}\rightarrow (\omega)^{k}$ be defined as follows: 
 
 \begin{equation}
 \pi(A)(i) =\begin{cases}
 A(i) & \text{if  $1\leq A(i)\leq k$}\\
 1 & \text{otherwise}.\end{cases}
  \end{equation}
 
 Notice that $\pi$ is a surjection. We understand $\pi$ as a projection function from $(\omega)^{\omega}$ onto $(\omega)^{k}$.  Following \cite{DM}, define an approximation function $s$ with domain  $\{0\}\cup\mathbb N\times (\omega)^{k}$ as follows. For  $i\in\{0\}\cup\mathbb N$ and  $X = \{X_1, \dots, X_k\}\in(\omega)^{k}$, the output $s(i,X)$ is a $k$-partition of some finite subset of $\mathbb N$ whose $j$-th class, $s(i,X)_j$ ($1\leq j\leq k$), will be defined by cases as follows:
 
 \medskip
 
\noindent For $i\leq k$, let 
  \begin{equation}\label{s-approx-1}
 s(i,X)_j = X_j\cap\{1, \dots, \min (X_k)\}
   \end{equation}

\noindent  So, for $i\leq k$, $ s(i,X)_j $ is the unique initial segment of $X_j$ included as a subset in $\{1, \dots, \min (X_k)\}$.  Note that for any $A\in (\omega)^{\omega}$ such that $\pi(A)=X$, if $i\leq k$ we have $s(i,X) = r(i,A)$. In particular, $s(0, X) = \emptyset$ and for $0<i\leq k$, $s(i,X)$ is a partition of $ \{1, \dots, \min (X_i)\}$.
 
 \medskip

 \noindent For $i>k$, let
  \begin{equation}\label{s-approx-2}
 s(i,X)_j = s(i-1,X)_j \cup \{\min (X_j \setminus s(i-1,X)_j)\}.
 \end{equation}
 
\noindent i.e., for $i>k$,  we obtain the class $s(i,X)_j $ by adding to $s(i-1,X)_j$ the minimum element of  $X_j$ that is not an element of $s(i-1,X)_j$. (Therefore, in this case,  $s(i,X)_j $ is as well an initial segment of $X_j$ like in the case $i\leq k$). This completes the definition of the approximation function $s$.  We think of $s(i,X)$ as the $i$-th approximation of $X$. Actually, each $X\in (\omega)^{k}$ can be identified with the sequence $(s(i,X))_{i\in\mathbb N}$ of its approximations.  

\medskip

Let $\mathcal S_k$ denote the range of $s$. We have 
 
 \begin{equation}
 (<\omega)^k\subset \mathcal S_k,\ \text{but } \mathcal S_k\setminus  (<\omega)^k\neq\emptyset.
  \end{equation}

   \bigskip
 
 Now, note that for every $a\in (<\omega)^k$ the union of $a$ is an initial segment of $\mathbb N$. That is not necessarily true for all elements of $\mathcal S_k$. Given $b\in\mathcal S_k$, if the union of $b$ is an initial segment of $\mathbb N$, then denote by $\#b$ the unique $l\in\mathbb N$ such that $b$ is a partition of $\{1, \dots, l\}$. Denote by $\cup b$ the union of $b$. Let 
 
  \begin{equation}
  \mathcal S_k^{\#} = \{b\in \mathcal S_k : \cup b \mbox{ is an initial segment of } \mathbb N\}.
  \end{equation}
 
 \medskip

 If  $a\in (<\omega)^k$ and  $b\in  \mathcal S_k$, write 
 
  \begin{equation} 
  a=s_k(b)
   \end{equation}  if there exist $X\in (\omega)^k$  and an integer $i > k$ such that $s(k,X)=a$ and $s(i,X)=b$. Now, let an integer $n$ with $k| n$ and $a\in (<\omega)^k$ be given. There exists a bijective correspondence between  $(n)^k_{\hom}$ and  the set  
 
   \begin{equation}
   T(a, k, n) =\{b\in \mathcal S_k^{\#} : a\neq b,\ a=s_k(b),\ \# b=\# a+n\}.
    \end{equation}

\noindent To see that such a correspondence exists, fix a bijective function  \begin{equation} \varphi : \{1, \dots, n\}\rightarrow \{\# a +1, \dots, \# a + n\}. \end{equation} Now, given $t\in (n)^k_{\hom}$, define $b_t\in T(a, k, n) $ by

\begin{equation}\label{b-approx}
 b_t(j) =\begin{cases}
a(j) & \text{if \ $1\ \leq j\leq \# a$}\\
t(\varphi^{-1}(j)) & \text{if $\# a < j\leq \# a +n$}\\
 \end{cases} 
 \end{equation}
 
 \noindent Then the correspondence   \begin{equation} t \mapsto b_t  \end{equation} is bijective. Obviously, the inverse of a partition $b\in  T(a, k, n)$ under this correspondence is the partition $t\in (n)^k_{\hom}$ defined by 
 
  \begin{equation}
  t(j) = b(\varphi(j)),\ \ 1\leq j\leq n.
   \end{equation}
 
 \medskip
 
From now on, for all positive integers  $k<n$ with $k|n$ and every $a\in (<\omega)^k$ we will fix one such bijective correspondence  $t \mapsto b_t$ between $(n)^k_{\hom}$ and $T(a, k, n)$ throughout the rest of this section. 

 \medskip
 
 Now we are ready to prove our main result.

\begin{proof}[Proof of Theorem \ref{hdrt}]
Fix positive integers $k<m$ and $N$,  with $k|m$, and suppose that the conclusion in the statement fails. For each positive integer $n>m$ such that $m|n$, choose a coloring $(n)^k_{\hom}=C^n_1\cup\dots\cup C^n_N$ admitting no monochromatic set of the form $(u)^k_{\hom}$ with $u\in (n)^m_{\hom}$.

\medskip

Given $X\in (\omega)^k$, let $a=s(k,X)$. Let 
 \begin{equation}
 i_0 =  \min\{ i>m : m \mbox{ divides } |\cup s(i,X)| - \# a \}.
 \end{equation} 
 
 \noindent Here, for $i>m$, $|\cup s(i,X)|$ denotes the cardinality of the union of the partition $s(i,X)$. Define the positive integer
 
  \begin{equation}
  n(X)=  |\cup s(i_0,X)| - \# a
   \end{equation} 
   
   \noindent Note that $k|n(X)$. List the elements of\  $\cup s(i_0,X)$ in their natural increasing order as\linebreak  $x_1< \dots < x_{\#a + n(X)}$, and let $b_X\in T(a,k,n(X))$ be defined by

\begin{equation}\label{b-approx}
 b_X(j) =\begin{cases}
a(j) & \text{if \ $1\ \leq j\leq \# a$}\\
s(i_0,X)(x_j) & \text{if $\# a < j\leq \# a +n(X)$}\\
 \end{cases} 
 \end{equation}
 
\noindent Denote by $t(b_X)$ the unique element $t\in (n(X))^k_{\hom}$ such that $b_t = b_X$. Define a coloring $(\omega)^k=C_1\cup\dots\cup C_N$ as follows: 

 \begin{equation}
 X\in C_j  \mbox{ if and only if } t(b_X)\in C^{n(X)}_j.
  \end{equation}

 \noindent Note that each $C_j$ is Borel. So by Theorem \ref{borel parameter words ramsey} there exists $A\in (\omega)^{\omega}$ such that $(A)^k\subseteq C_{j_0}$ for some $j_0\in\{1, \dots, N\}$. Fix $a_0\in (<\omega)^{m}\upharpoonright A$ and let $Y\in (A)^m$ be such that $s(m,Y)=a_0$. Let 
 
  \begin{equation}
  l =  \min\{ i>m : m \mbox{ divides } |\cup s(i,Y)| - \# a_0 \} \end{equation}  and set $n= |\cup s(l,Y)| - \# a_0$.  List the elements of\  $\cup\, s(l,Y)$ in their natural increasing order as $y_1< \dots < y_{\#a_0 + n}$. Define $b_Y\in T(a_0,m,n)$ by

\begin{equation}\label{b-approx}
 b_Y(j) =\begin{cases}
a_0(j) & \text{if \ $1\ \leq j\leq \# a_0$}\\
s(l,Y)(y_j) & \text{if $\# a_0 < j\leq \# a_0 +n$}\\
 \end{cases} 
 \end{equation}

\noindent Now denote by $u$ the unique element $t\in (n)^m_{\hom}$ such that $b_t = b_Y$.

Finally, for $t\in (u)^k_{\hom}$, define $X\in (A)^k$ with $X\leq Y$ by merging  classes in $Y$ according to how the corresponding classes in $u$ were merged to build $t$. Let  $a = s(k, X)$ and note that  $n(X) = n$. If $b_t$ is the unique element of $T(a,k,n)$ corresponding to $t$, then $b_X=b_t$ and therefore $t=t(b_X)$. Thus, $t\in C^n_{j_0}$, by the choice of $A$. Since $t$ was arbitrary, we get  $(u)^k_{\hom}\subseteq C^n_{j_0}$. A contradiction. This completes the proof.
\end{proof}

\end{document}